\documentclass[11pt]{article}
\usepackage{geometry} 
\usepackage{amsmath,amsthm,amssymb}                
\usepackage{graphicx, color}
\usepackage{amssymb}
\usepackage{epstopdf}
\usepackage{pdfsync}
\usepackage{mathabx}

\newtheorem{definition}{Definition}[section]

\newtheorem{theorem}[definition]{Theorem}
\newtheorem{proposition}[definition]{Proposition}

\newtheorem{remark}[definition]{Remark}
\newtheorem{example}[definition]{Example}

\numberwithin{equation}{section}

\def\e{\varepsilon}

\newcommand{\LM}[2]{\hbox{\vrule width.4pt \vbox to#1pt{\vfill
\hrule width#2pt height.4pt}}}
\newcommand{\LLL}{{\mathchoice {\>\LM{7}{5}\>}{\>\LM{7}{5}\>}{\,\LM{5}{3.5}\,}{\,\LM{3.35}{2.5}\,}}}

\title{Ising systems, measures on the sphere, and zonoids\\
}
\author{Andrea Braides
\\
SISSA
\\
via Bonomea 265\\
34146 Trieste, Italy\\
{email: abraides@sissa.it}
 \and  Antonin Chambolle
\\
 CEREMADE\\
CNRS \& Universit\'e de Paris-Dauphine
\\
   Place de Lattre de Tassigny
\\
  75775 Paris,
France\\
{email: antonin.chambolle@ceremade.dauphine.fr}
}
\date{}
\begin{document} 

\maketitle

\begin{abstract} We give an interpretation of a class of discrete-to-continuum results for Ising systems using the theory of zonoids. We define the classes of {\em rational zonotopes and zonoids}, as those of the Wulff shapes of perimeters obtained as limits of finite-range homogeneous Ising systems and of general homogeneous Ising systems, respectively. Thanks to the characterization of zonoids in terms of measures on the sphere, rational zonotopes, identified as finite sums of Dirac masses, are dense in the class of all zonoids. Moreover, we show that a rational zonoid can be obtained from a coercive Ising system if and only if the corresponding measure satisfies some `connectedness' properties, while it is always a continuum limit of `discrete Wulff shapes' under the only condition that the support of the measure spans the whole space. Finally, we highlight the connection with the homogenization of periodic Ising systems and propose a generalized definition of rational zonotope of order $N$, which coincides with the definition of rational zonotope if $N=1$.

%
%
\end{abstract}

\section{Introduction}
We consider Ising systems; that is, energies depending on a {\em spin parameter}, formally written as
\begin{equation}\label{Ising-1}
-\sum_{i\neq j} \alpha_{i-j} u_i u_j.
\end{equation} 
In this notation the spin functions $u\colon \Omega\cap \mathbb Z^d \to \{-1,1\}$ is defined on the portion of the standard cubic lattice contained in the (bounded) open set $\Omega$, and we write $u_i$ in the place of $u(i)$.  
The systems are supposed to be {\em ferromagnetic}; that is, $\alpha_k\ge 0$ for all $k\in \mathbb Z^d$. This condition implies that the interactions between nodes at distance $k$ such that $\alpha_k>0$ are minimal if $u_i=u_j$. This condition in turn implies that the only {\em ground states} are the constant states, provided that the network of interaction is {\em connected}, in the sense that for every $j\in\mathbb Z^d\setminus\{0\}$ there exist $K\in\mathbb N$ and $k_1,\ldots, k_K$ such that $k=k_1+\cdots+k_K$ and $\alpha_{k_\ell}>0$ for all $\ell\in\{1,\ldots, K\}$, and also $\Omega$ is likewise connected. While the form \eqref{Ising-1} is quite descriptive, it is convenient for our purposes to consider an equivalent energy of the form
\begin{equation}\label{Ising-2}
\sum_{i\neq j} \alpha_{i-j} (u_i -u_j)^2.
\end{equation}
Indeed, since $u_j^2=u_i^2=1$ developing the square, the terms in \eqref{Ising-2} can be rewritten as
$$
\sum_{i\neq j} \alpha_{i-j} (u_i -u_j)^2= 2\sum_{i\neq j} \alpha_{i-j} -2\sum_{i\neq j}\alpha_{i-j}u_iu_j,
$$
and $2\sum_{i\neq j} \alpha_{i-j}$ is a constant depending only on $\Omega$. In the form \eqref{Ising-2}, ground states have always zero energy and we can also consider $\Omega$  unbounded since we  thus avoid annoying $+\infty-\infty$ indeterminate forms. In this paper $\Omega$ plays no role, and we take $\Omega=\mathbb R^d$ for simplicity. Furthermore, the lattice $\mathbb Z^d$ can be substituted with a Bravais lattice, at the expense only of a heavier notation in some proofs.

The overall behaviour of systems \eqref{Ising-2} can be described by introducing an {\em effective surface tension}, which takes the form
\begin{equation}\label{Ising-3}
\varphi(\nu) =  4\sum_{k\in\mathbb Z^d} \alpha_{k} |\langle k,\nu\rangle|,
\end{equation}
which describes the energy density of a minimal interface macroscopically oriented as an hypersurface with normal $\nu$. This effective surface tension can be obtained in various ways, which describe different ways of looking at the problem. One way is to compute the average limit behaviour of minimum problems on large cubes with two faces orthogonal to $\nu$ and boundary conditions jumping in correspondence of the mid-plane of the cube; a more complete analysis is obtained by looking at minimizers of problems in the whole space with the volume constraint $\#\{u_i=1\}= N$, and prove that, upon suitably scaling and translating  them, they converge (after suitable interpolation) to minimizers of the perimeter energy related to $\varphi$. A relatively recent way to explain this convergence is by a {\em discrete-to-continuum} approach, as a result of a limit analysis for the scaled energies
\begin{equation}\label{Ising-4}
E_\e(u)=\sum_{i\neq j}\e^{d-1} \alpha_{i-j} (u_i -u_j)^2,
\end{equation}
where the scaled {\em spin parameter} $u\colon \Omega\cap\e \mathbb Z^d \to \{-1,1\}$ is now defined on the scaled standard cubic lattice, and we write $u_i$ in the place of $u(\e i)$. Each function $u$ is extended as a piecewise-constant function, so that the domain of each $E_\e$ can be identified as a subset of $L^1_{\rm loc}(\Omega;\{-1,1\})$,
and, if the lattice $\mathbb Z^d$ is connected with respect to $\{\alpha_k\}$ in the sense above, the family $E_\e$ is equicoercive with respect to the strong convergence in $L^1_{\rm loc}(\Omega;\{-1,1\})$, so that a family $u^\e$  with equibounded energy converges, up to subsequences, to a continuum parameter $u$ with $u(x)\in\{-1,1\}$ almost everywhere. If we write $u= 2\chi_A-1$, this defines a discrete-to-continuum convergence of spin functions $u^\e$ to sets $A$, which are indeed sets of finite perimeter. With respect to this convergence, the $\Gamma$-limit of the functionals $E_\e$ is the (anisotropic) perimeter functional
\begin{equation}\label{Ising-5}
F(A)=\int_{\Omega\cap\partial^*A}\varphi(\nu)d{\mathcal H}^{d-1},
\end{equation}
where $\partial^*A$ and $\nu_A$ are the {\em reduced boundary} and the measure-theoretical {\em internal normal} to $\partial^*A$, respectively, and $\varphi$ is defined by \eqref{Ising-3}. In this form, the result has been proved in various versions, first in the context of free-discontinuity problems by Chambolle \cite{Ch-2} and by Braides and Gelli \cite{BGe}, whose proof is then reset in terms of Ising systems in \cite[Section 3.1]{ABCS} (for a simplified exposition in a two-dimensional context see also Section 3.2.4 in the book by Braides and Solci \cite{GeomFlow}). This analysis can be seen as a homogenization problem with a $1$-periodic system of interactions \cite{BP} and as such $\varphi$ can be defined via an asymptotic homogenization formula, reducing to the limit analysis of minimum problems on cubes as described above. Conversely, the convergence of minimum problems with volume constraints to a {\em Wulff problem} for the perimeter $F$ is a consequence of the property of convergence of minima of $\Gamma$-convergence.

Scope of this paper is to connect these variational descriptions of Ising systems with
the concept of {\em zonoid} from Convex Geometry. 
 A zonoid is defined as a limit in the Hausdorff metric of {\em zonotopes}, which are simply defined to be vector sums of a finite number of line segments. As such, their support functions can be
written in the form
\begin{equation}\label{sufu}
f(z)=\sum_{j=1}^N m_j|\langle \nu_j,z\rangle|,
\end{equation}
where $\nu_i\in S^{d-1}$ and $m_j>0$. Comparing \eqref{sufu} with \eqref{Ising-3} we note that the latter requires some restrictions on $\nu_j$. With this observation in mind, we then define the subclass of {\em rational zonotopes} as those for which all $\nu_j$ in \eqref{sufu} are rational directions; i.e., $\nu_j={k_j\over\|k_j\|}$ for some $k_j\in\mathbb Z^d\setminus\{0\}$. Hence, an effective surface tension $\varphi$ in \eqref{Ising-3}  for a system $\{\alpha_k\}$ with $\alpha>0$ only for a finite set of $k\in\mathbb Z^d$ can be interpreted as the support function of a rational zonotope. Note that the family of rational zonotopes is still a dense class in the family of zonoids. 

The fundamental property for the analysis of zonoids is that they can be identified with (symmetric) positive bounded  measures $\mu$ on $S^{d-1}$ such that the support function $f$ of the zonoid can be written as
 \begin{equation}\label{sufu-measure}
f(z)=\int_{S^{d-1}} |\langle \nu,z\rangle|\,d\mu(\nu).
\end{equation}
In the case of zonotopes this measure is a finite sum of Dirac deltas, which, in case of rational zonotopes, are concentrated on a set of rational directions. We then define the class of  {\em rational zonoids} as that of the zonoids corresponding to possibly infinite sums of Dirac deltas concentrated on rational directions. This is the class corresponding to general $\varphi$ in \eqref{Ising-3}.

\smallskip
We then have the following characterizations.

{\em Exact reachability}. The functionals $F$ obtained as limits of Ising systems are all functionals whose Wulff shapes are rational zonoids.

{\em Approximate reachability}. For each zonoid there exists a family $F_n$ of functionals obtained as limits of Ising systems such that the corresponding Wulff shapes converge to the zonoid.

{\em Convergence of discrete Wulff shapes.} If a rational zonoid has positive Lebesgue measure then there exists an Ising system with constrained minimizers, suitably identified with sets, that approximate the zonoid.
For lower-dimensional rational zonoids the same holds in a lower-dimensional subspace.

{\em Coercive Ising systems.} We highlight a connectedness property of the generating measure of a rational zonoid which is necessary and sufficient for the existence of a corresponding coercive Ising system.

\smallskip

We further note that we do not have uniqueness of generating Ising systems, in the sense that the same rational zonoid corresponds to infinitely many equivalent Ising systems, for some of which we may not have the property of convergence of discrete minimizers. 

The concept of zonoid has been generalized in various ways (see \cite[Chapter 9]{Sch}). The variational interpretation of rational zonoids allows to view them as a particular case of homogenization of periodic Ising systems when the period is $1$. With this observation in mind, we finally propose a generalization of rational zonotopes and zonoids as those obtained by homogenization of periodic Ising systems. If such a system of period $N$ is of finite range, the Wulff shape of the corresponding perimeter is a polytope, due to the results by Chambolle and Kreutz \cite{CK}, which allows to define rational zonotopes of order $N$.
The closure of all zonotopes of order $N$ with varying $N$ is proved to be the set of all convex centered sets using the results of Braides and Kreutz \cite{BK}.

\section{Zonotopes, zonoids and their support functions}

In the following two sections we recall some definitions and properties from the theory of zonoids, for which we refer to the monograph by Schneider \cite{Sch}. In Section \ref{ratzo} we introduce the subclass of rational zonoids.

\subsection{Zonoids}
A (centered) {\em zonotope} in $\mathbb R^d$ is a polytope that is obtained as a Minkowski sum of a finite number of centered segments $[-w_i,w_i]$ with $w_i\in \mathbb R^d$, $i\in\{1,\ldots,N\}$, and $N\in\mathbb N$; that is, a set of the form
\begin{equation}\label{zonoidW}
W=\Bigl\{w\in\mathbb R^d\colon  \hbox{ there exist } s_i\in[-1,1]\hbox{ such that } w=\sum_{i=1}^N s_iw_i\Bigr\}.
\end{equation}
The usual definition of zonoid (see \cite{Sch}) does not require that the segments be centered. However, any (general) zonoid is the translation of a centered zonoid. Since we will mainly deal with symmetric sets in $\mathbb R^d$ we directly use centered zonoids in order to simplify the notation and terminology.

We say that $W$ is {\em non-degenerate} if the vectors $w_1,\ldots, w_N$ span the whole $\mathbb R^d$ so that $W$ is a convex set symmetric with respect to the origin and of non zero Lebesgue measure. If otherwise, a degenerate zonotope can be identified as a non-degenerate zonotope in a lower-dimensional space. 

It is worth noting that zonotopes are particular centered symmetric polytopes characterized by the fact that their faces are themselves (congruent to $d-1$-dimensional) zonotopes. This property rules out a number of polytopes in dimension $d\ge 3$; e.g.~octahedrons.

Using \eqref{zonoidW}, the {\em support function} of a zonotope is then given by
$$
f_W(z)=\sup\{\langle z,w\rangle: w\in W\}= \sum_{i=1}^N |\langle z,w_i\rangle|.
$$
Conversely, given $f$ of this form, the set $W$ in \eqref{zonoidW} coincides with the {\em Wulff shape} of $f$, given by
\begin{equation}\label{WuSh}
W_f=\Big\{w \in\mathbb R^d: \langle z,w\rangle\le 1\hbox{ for all } z \in\mathbb R^d\hbox{ such that } f(z)\le 1\Big\}.
\end{equation}

The family of (centered) {\em zonoids} in $\mathbb R^d$ is the family of all convex symmetric sets that can be obtained as limits of zonotopes in the Hausdorff metric. We say that a zonoid is non-degenerate if it has a non empty interior, in which case it is the limit of non-degenerate zonotopes. Note that in dimension $d=2$ all convex symmetric sets are zonoids, while the symmetry restrictions on the faces of zonotopes imply that zonoids are nowhere dense in the family of  all convex symmetric sets if $d\ge 3$.

\subsection{Generating measures and support functions of zonoids}
For a zonotope $W$ as in \eqref{zonoidW}, after setting $\nu_i={w_i\over\|w_i\|}$, we can write
\begin{eqnarray*}
f_W(z)&=&\sum_{i=1}^N |\langle z,w_i\rangle|\\
&=& \sum_{i=1}^N |\langle z,\nu_i\rangle|\,{\|w_i\|}=\int_{S^{d-1}}|\langle z,\nu\rangle|\,{\|w_i\|}\,d\Bigl({\delta_{\nu_i}+\delta_{-\nu_i}\over 2}\Bigr)(\nu)
\\
&=&\int_{S^{d-1}}|\langle z,\nu\rangle| d\mu_W(\nu),
\end{eqnarray*}
where
$$
\mu_W=\sum_{i=1}^N{\|w_i\|\over 2}\bigl(\delta_{\nu_i}+\delta_{-\nu_i}\bigr).
$$ 
Conversely, given a positive measure of the form $\mu=\sum_{i=1}^N{\lambda_i}(\delta_{\nu_i}+\delta_{-\nu_i})$, 
with $\nu_i\in S^{d-1}$, setting
$$
f_\mu(z)=\int_{S^{d-1}}|\langle z,\nu\rangle|\,d\mu(\nu)=\sum_{i=1}^N 2\lambda_i |\langle z,\nu_i\rangle|,
$$
we have that $f_\mu= f_W$, where $W$ is given by \eqref{zonoidW} with $w_i=2\lambda_i\nu_i$. Hence, zonotopes correspond to (symmetric) linear combinations of Dirac deltas on $S^{d-1}$ with positive coefficients.
Note that the Hausdorff convergence of zonoids corresponds to the weak$^*$ convergence of the related measures. By the weak$^*$ density of sums of Dirac deltas this shows that positive symmetric measures on $S^{d-1}$ are in bijection with zonoids.

The support functions of (centered) zonoids in $\mathbb R^d$ are characterized by elements of the cone of positive symmetric measures on $S^{d-1}$ as in the following proposition.

\begin{proposition}
For every (centered) zonoid $W$ in $\mathbb R^d$ there exists a unique symmetric positive measure $\mu_W$ on $S^{d-1}$ such that the support function $f_W$ can be written as 
\begin{equation}
f_W(z)=\int_{S^{d-1}}|\langle z,\nu\rangle|\,d\mu_W(\nu).
\end{equation}
\end{proposition}

Such a measure is called the {\em generating measure} of $W$.

\section{Ising systems and a variational interpretation of rational zonotopes and zonoids}

%
%
%
%
%

We consider homogeneous systems of discrete interactions governed by energies of the form
\begin{equation}\label{fegen}
E_\e(u)=\sum_{i,j\in\mathbb Z^d}\e^{d-1} \,\alpha_{i-j}(u_i-u_j)^2,
\end{equation}
defined on functions $u\colon \e\mathbb Z^d\to \{-1,1\}$, where we use the notation $u_i=u(\e i)$.
Note that we can assume, and we will, that $\alpha_{-k}=\alpha_{k}$ since otherwise we can replace both coefficients by $\frac{\alpha_{k}+\alpha_{-k}}{2}$, and this change does not influence the value of the energy. 
We assume that the system is {\em ferromagnetic}; that is, $\alpha_{k}\geq 0$ for any $k\in \mathbb Z^d$.
We further assume the {\em decay condition}
\begin{equation}\label{decay}
\sum_{k\in\mathbb Z^d} \alpha_k\|k\|<+\infty. 
\end{equation}
This condition is necessary to have non-trivial energies, in the sense that if this condition fails then the limit of $E_\e$ as defined below is finite only if $u$ is identically $1$ or $-1$. The convex function $\varphi\colon \mathbb R^d\to[0,+\infty)$
\begin{equation}\label{fi}
\varphi(z)= 4\sum_{k\in\mathbb Z^d} \alpha_k|\langle z,k\rangle|,
\end{equation}
is well defined and finite thanks to \eqref{decay}.

\subsection{Rational zonoids}\label{ratzo}
The particular form of the functions $\varphi$ in \eqref{fi} suggests a definition of a class of zonoids, of which such types of functions are support functions.

\begin{definition}
We say that $\nu\in S^{d-1}$ is a {\em rational direction} if there exist $w\in\mathbb Z^d\setminus \{0\}$ such that
$$
\nu={w\over\|w\|}.
$$
A set $W$ is a {\em (centered) rational zonotope} if its generating measure is of the form
\begin{equation}
\mu_W=\sum_{i=1}^N \lambda_i (\delta_{\nu_i}+\delta_{-\nu_i}),
\end{equation}
where $N\in\mathbb N$, $\nu_i$ are rational directions and $\lambda_i>0$. A set is a {\em (centered) rational zonoid} if there exists a sequence $\{\nu_i\}$ of rational directions and a summable sequence $\{\lambda_i\}$ of positive numbers such that
\begin{equation}
\mu_W=\sum_{i=1}^{+\infty} \lambda_i (\delta_{\nu_i}+\delta_{-\nu_i}).
\end{equation}
\end{definition}

\begin{remark}\rm
By the density of rational directions in $S^{d-1}$, rational zonotopes (and hence also rational zonoids) are dense in the class of zonoids.\end{remark}

\subsection{Sets of finite perimeter and their energies}
To each Ising system we will associate a {\em perimeter energy}. To that end we recall that a subset $A$ in $\mathbb R^d$ is a {\em set of finite perimeter} if the distributional gradient of its characteristic function $\chi_A$ is a bounded measure. We refer to \cite{AFP,B-LN,Maggi} for an introduction to the topic. Here we only recall that if $A$ is set of finite perimeter there exists a Borel set $\partial^*A$, the {\em reduced boundary of $A$}, and a function $\nu=\nu_A\colon \partial^*A\to S^{d-1}$, the {\em inner normal to $A$}, such that $D\chi_A=\nu \mathcal H^{d-1}\LLL \partial^*A$. Furthermore if $\varphi\colon \mathbb R^d\to[0,+\infty)$ is a convex function positively homogeneous of degree one, then the {\em perimeter energy} $F=F_\varphi$ defined by
\begin{equation}\label{perimeterF}
F_\varphi(A)=\int_{\partial^*A}\varphi(\nu(x))d\mathcal H^{d-1}(x)
\end{equation}
is weakly lower semicontinuous with respect to the convergence of $\chi_{A_\e}$ to $\chi_A$ in $L^1_{\rm loc}(\mathbb R^d)$ on families such that the total variations $\|D\chi_{A_\e}\|$ are equibounded.

We finally recall that families of sets of finite perimeter such that $\|D\chi_{A_\e}\|$ are equibounded are precompact with respect to the convergence $\chi_{A_\e}\to \chi_A$ in $L^1_{\rm loc}(\mathbb R^d)$, and that if $A$ is a set of finite perimeter then there exists a family of polyhedral sets $A_\e$ converging in the sense above to $A$  as $\e\to 0$, and such that $F_\varphi(A_\e)$ tends to $F_\varphi(A)$. Note that for polyhedral sets we have $\partial^*A=\partial A$, up to an $\mathcal H^{d-1}$-negligible set.

\subsection{Convergence of the scaled energies of an Ising system}

We say that a sequence $u^\e\colon \e\mathbb Z^d\to \{-1,1\}$ {\em converges to a set of finite perimeter} $A$ if the piecewise-constant interpolations $u_\e$ of $u^\e$ on $\e\mathbb Z^d$, defined by $u_\e(x)=u^\e_i\ (=u^\e(\e i))$ if $x\in \e i+[0,\e)^d$, locally converge in $L^1(\mathbb R^d)$ to the function $u=2\chi_A-1$ and there exists $C$ such that $\|Du_\e\|=\|Du_\e\|(\mathbb R^d)\le C$, where $\|Du_\e\|$ denotes the total variation of the measure $Du_\e$. In other words, the sequence $u_\e$ converges weakly in
$BV_{\rm loc}(\mathbb R^d)$.

\bigskip
The condition that $\|Du_\e\|(\mathbb R^d)\le C$ is a consequence of the boundedness of the energies $E_\e(u^\e)$ if there holds a condition of the type
\begin{equation}\label{coercnn}
\hbox{$\alpha_{k}\geq c>0$ if $\|k\|=1$ ({\em coerciveness of nearest-neighbour interactions}).}
\end{equation}
\noindent In this case, we have that there exists $A$ such that $u^\e\to A$ up to subsequences. A thorough description of the limit of Ising systems in terms of perimeter functionals when condition \eqref{coercnn} is satisfied is given in \cite[Section 3.1]{ABCS}.

\smallskip
A general necessary and sufficient condition for coerciveness will be given below.
Note however that we do not make any such assumption in our definition of convergence of $E_\e$, in order to include also degenerate cases in our treatment.

\begin{theorem}[limits of homogeneous Ising systems as rational zonoids]\label{teo-hom-sys} 
\index{limits of homogeneous systems} A functional of the form \eqref{perimeterF}
is a $\Gamma$-limit  with respect to the convergence $u^\e\to A$ of energies $E_\e$ of the form \eqref{fegen} for some Ising system $\{\alpha_k\}$ with $\alpha_k\ge 0$ satisfying \eqref{decay}
if and only if the Wulff shape of $\varphi$ is a rational zonoid. Furthermore, if the range of $\{\alpha_k\}$ is finite the Wulff shape of $\varphi$ is a rational zonotope.
\end{theorem}

\begin{proof}
Given $\{\alpha_k\}$ non-negative coefficients satisfying \eqref{decay} with respect to the convergence $u^\e\to A$ , the $\Gamma$-limit of the sequence of energies $E_\e$ defined by is given by an energy \eqref{perimeterF} with $\varphi$ given by \eqref{fi}. We briefly give a proof. This can also be achieved with a perturbation argument from the analog result for coercive functionals, for which we refer e.g.~to Section 3.1 in \cite{ABCS}, where the interested reader can find further common details with the proof presented below. 

In order to provide a lower bound, we examine separately energies with only the contribution for a fixed $k\in\mathbb Z^d\setminus\{0\}$. We can suppose that the last component $k_d$ be strictly positive, and define the lattice $\mathcal L=\mathcal L_k$ as the Bravais lattice generated by $\{e_1,\ldots, e_{d-1}, k\}$, which is a sub-lattice of $\mathbb Z^d$. 
We consider the functionals
$$
E^k_\e(u)=\sum_{i\in \mathcal L}\e^{d-1}(u_{i+k}-u_i)^2.
$$
These can be seen as a system of nearest-neighbour interactions on the lattice $\mathcal L$ with $\alpha_{e_j}=0$ for $j\in\{1,\ldots, d-1\}$.

Let $u^\e\to A$ with $\sup_\e \|Du_\e\|<+\infty$.
For every $u^\e$ we can define its interpolation $u^{\mathcal L}_\e$ on the lattice $\e\mathcal L$ defined by $u^{\mathcal L}_\e(x)= u^\e_i$ if $x\in \e i+\e U$, where $U$ is the $d$-dimensional parallelogram with sides $e_1,\ldots, e_{d-1}, k$ and $i\in\mathcal L$. Note that $\sup_\e \|Du^{\mathcal L}_\e\|<+\infty$, so that we can suppose that $u^{\mathcal L}_\e\to 2\chi_{A^{\mathcal L}}-1$ in $L^1_{\rm loc}(\mathbb R^d)$, for some set of finite perimeter $A^{\mathcal L}$. Since $u_\e(x)=u^{\mathcal L}_\e(x)$ on $\e \mathcal L +\e (U\cap [0,1)^d)$ we then deduce that $A^{\mathcal L}=A$. If we define $A^{\mathcal L}_\e=\{ x\in \mathbb R^d: u^{\mathcal L}_\e(x)=1\}$ then we have 
$$
E^k_\e(u^\e)= {\|k\|\over k_d}4\int_{\partial A^{\mathcal L}_\e}\Big|\Big\langle \nu,{k\over \|k\|}\Big\rangle\Big| d{\mathcal H}^{d-1}
= {1\over k_d}4\int_{\partial A^{\mathcal L}_\e}|\langle \nu,k\rangle| d{\mathcal H}^{d-1},
$$
where we have taken into account that the parts of the boundary of $A^{\mathcal L}_\e$ with normal $\nu\in\{e_1,\ldots, e_{d-1}\}$ do not contribute to the energy, and the projection of $U$ on the hyperplane orthogonal to $k$ has $d-1$-measure equal to ${k_d\over \|k\|}$. Note that in this case $E^k_\e(u^\e)$ equals the total variation $\|D_ku^{\mathcal L}_\e\|$, where $D_k$ denotes the distributional directional derivative in the direction $k$.

Taking into account the lower semicontinuity of this perimeter functional, we then have
$$
\liminf_{\e\to 0}E^k_\e(u^\e)\ge {1\over k_d}4\int_{\partial^*A}|\langle \nu,k\rangle| d{\mathcal H}^{d-1}.
$$
Since $|U|=k_d$ the number of equivalence classes of $\mathbb Z^d$ modulo $\mathcal L$ is $k_d$, from which (proceeding as above in each of these equivalence classes) we have
$$
\liminf_{\e\to 0} \sum_{i\in \mathbb Z^d}\e^{d-1}(u_{i+k}-u_i)^2\ge 
k_d\liminf_{\e\to 0} E^k_\e(u^\e)\ge 4\int_{\partial^*A}|\langle \nu,k\rangle| d{\mathcal H}^{d-1}.
$$
From these inequalities, valid for all $k\in\mathbb Z^d$, the inequality $\liminf\limits_{\e\to 0} E_\e(u^\e)\ge F(A)$ follows.
To prove the upper bound it suffices to note that if $A$ is a polyhedron then the restriction of $u^\e=2\chi_A-1$ to $\e\mathbb Z^d$ is a recovery sequence satisfying $\|Du_\e\|\le C<+\infty$. The proof of the $\Gamma$-convergence is then completed by an approximation argument.

Since the function $\varphi$ is the (locally uniform) limit of the functions
$$
\varphi_n(z)= 4\sum_{k\in\mathbb Z^d,\ \|k\|\le n}\alpha_k |\langle k,z \rangle|,
$$
whose Wulff shapes are rational zonotopes corresponding to the measures on $S^{d-1}$
$$
\mu_n=\sum_{k\in\mathbb Z^d,\ \|k\|\le n}4\alpha_k\|k\|\delta_{ {k\over\|k\|}},
$$
$\varphi$ is the support function of a rational zonoid corresponding to 
$$
\mu=\sum_{k\in\mathbb Z^d}4\alpha_k\|k\|\delta_{ {k\over\|k\|}}.
$$
All these measures are symmetric since we assume $\alpha_k=\alpha_{-k}$.
Note that if $\{\alpha_k\}$ is of finite range then $\varphi$ is the support function of a rational zonotope.

Conversely, if we have a rational zonoid $W$ corresponding to a finite symmetric positive measure 
$$
\mu_W=\sum_{i}\beta_i \big(\delta_{ \nu_i}+\delta_{ -\nu_i}\big),
$$
with $\nu_i\in S^{d-1}$ rational directions. 
Note that 
$$
\int_{S^{d-1}}|\langle z,\nu\rangle|d(\delta_{ \nu_i}+\delta_{ -\nu_i}\big)(\nu)=2|\langle z,\nu_i\rangle|, 
$$
so that 
$$
f_W(z)=\int_{S^{d-1}}|\langle z,\nu\rangle|d\mu_W(\nu)=2\sum_{i}\beta_i|\langle z,\nu_i\rangle|. 
$$
Then for all $i$ we can fix $k_i\in \mathbb Z^d$ such that $\nu_i={k_i\over \|k_i\|}$ and define
$$\alpha_k=\begin{cases}{\beta_i\over 4\|k_i\|}&\hbox{if $k=k_i$ or $k=-k_i$ for some $i$}\\
0 &\hbox{otherwise,}\end{cases}
$$
so that we have
$$
\varphi(z)= 4\sum_{k\in\mathbb Z^d}\alpha_k |\langle k,z\rangle|=
2\sum_{i}\beta_i \Big|\Big\langle z,{k_i\over \|k_i\|}\Big\rangle\Big|,
$$
and $\sum_k\alpha_k\|k\|={1\over 2}\sum_i\beta_i<+\infty$, so that \eqref{decay} is satisfied.
\end{proof}

\begin{definition}
We say that an Ising system $\{\alpha_k\}$ {\em generates a rational zonoid $W$}, or equivalently it {\em generates a measure} $\mu_W$ (the generating measure of $W$), or equivalently it {\em generates an energy density} $f_W$, the support function of $W$, if we have that $E_\e$ $\Gamma$-converges to $F$ in the sense of Theorem {\rm \ref{teo-hom-sys}} with $\varphi=f_W$. We say that two Ising systems as above are {\em equivalent} if they generate the same zonoid.
\end{definition}

\begin{remark}[equivalent Ising systems]\label{no-inj}\rm
For every rational direction $\nu\in S^{d-1}$ let ${\mathcal I}(\nu)=\big\{k\in\mathbb Z^d\setminus \{0\}: {k\over\|k\|}=\nu\big\}$. 
We can rewrite formula \eqref{fi} as
\begin{equation}\label{fi-1}
\varphi(z)= 4\sum_{\nu}\sum_{k\in {\mathcal I}(\nu)} \alpha_k|\langle z,k\rangle|= 4\sum_{\nu}\Big(\sum_{k\in {\mathcal I}(\nu)} \alpha_k\|k\|\Big)|\langle z,\nu\rangle|.
\end{equation}
From \eqref{fi-1} and taking the symmetry of $\alpha_k$ into account we note that two Ising systems satisfying \eqref{decay} generate the same rational zonoid if and only if for every rational direction $\nu\in S^{d-1}$ we have
\begin{equation}
\sum_{k\in {\mathcal I}(\nu)}\alpha_k\|k\|=\sum_{k\in {\mathcal I}(\nu)}\alpha'_{k}\|k\|.
\end{equation}

As an example, we may take the systems (parameterized on sequences $\{\lambda_n\}$)
\begin{equation}\label{alfamult}
\alpha_k=\begin{cases} 
\lambda_{|n|} &\hbox{ if }k=n e_\ell, n\in\mathbb Z, \ \ell\in\{1\ldots, d\}\\
0 &\hbox{ otherwise,}
\end{cases}
\end{equation}
with $\sum_{n=1}^\infty n\lambda_n=\lambda$. Then $\varphi(z)=4\lambda\sum_{n=1}^d|z_n|=:4\|z\|_1$, and the corresponding $W$ is the same coordinate square depending only on $\lambda$ and not on the particular sequence. 

Let $\mu$ be a measure generated by the system $\{\alpha_k\}$. Note that if $\nu\in S^{d-1}$ is such that $\mu(\{\nu\})>0$ then the set of indices $k\in {\mathcal I}(\nu)$ such that $\alpha_k>0$ may be infinite even though $\{\alpha_k\}$ generates a rational zonotope.
Conversely, for all rational zonoid $W$ there exists an Ising system $\{\alpha_k\}$ generating $W$ such that for all $\alpha_k>0$ $k\neq 0$ such that $\mu_W\big({k\over\|k\|}\big)>0$. Indeed, it suffices to note that, if $\alpha(\nu)=\sum_{k\in \mathcal I(\nu)}\alpha_k\|k\|$, then an equivalent Ising system is $\{\alpha'_k\}$ given by
$\alpha'_k=2^{-n}{\alpha(\nu)\over n\|k_0(\nu)\|}$ if $k\in \mathcal I(\nu)$ and $k=nk_0(\nu)$, where $k_0$ is the element of least norm in $\mathcal I(\nu)$.
\end{remark}

 The following definition generalizes condition \eqref{coercnn}.

\begin{definition}\label{coer-def} 
We say that an Ising system $\{\alpha_k\}$ is a {\em coercive system} if there exists a constant $M>0$ such that 
\begin{equation}\label{coerFun}
\|Du\|\le M E(u)\ \hbox{, where }\quad E(u)=\sum_{i,j\in\mathbb Z^d} \alpha_{i-j} (u_i-u_j)^2,
\end{equation}
where in the left-hand side we have identified $u\colon \mathbb Z^d\to\{-1,1\}$ with its piecewise-constant extension from $\mathbb Z^d$.
\end{definition}

\begin{remark}\rm
Note that in \eqref{coerFun}  
\begin{equation}
\|Du\|= 4\mathcal H^{d-1}(\partial\{x: u(x)=1\})= 2\#\{ (i,j): u_i\neq u_j\},
\end{equation}
the factor $4$ coming from the fact that $(u_i-u_j)^2=4$ if $u_i\neq u_j$, and the factor $2$ coming from counting both $(i,j)$ and $(j,i)$ if $u_i\neq u_j$. From the last equality we have that \eqref{coercnn} implies that $\{\alpha_k\}$ is coercive. 

From the definition of $E_\e$ we obtain that \eqref{coerFun} is equivalent to $\|Du_\e\|\le M E_\e(u^\e)$  with $M$ independent of $\e$, so that
 if $\sup_\e E_\e(u^\e)<+\infty$ then also $\sup_\e \|Du_\e\|<+\infty$ and then, up to subsequences, $u^\e\to A$ for some set of finite perimeter $A$.
\end{remark}

\begin{remark}[equivalent coercive and non-coercive systems]\label{co-si} \rm
In the assumptions of Theorem \ref{teo-hom-sys}, in general the sequence $E_\e$ is not coercive; that is, we
cannot deduce that there exists $A$ such that $u^\e\to A$ up to subsequences from the boundedness of the energies $E_\e(u^\e)$. In the case that \eqref{coercnn} holds we have a subclass of $\varphi$, for which $\alpha_k>0$ if $k\in\{e_1\ldots, e_d\}$, and the construction in the proof of the theorem gives coercive approximating $E_\e$. Note however that from the form of $\varphi$ we cannot deduce the equicoerciveness of $E_\e$. Indeed, let $\varphi(z)=4\|z\|_1$, for which we may take (see \eqref{alfamult} in Remark \ref{no-inj})
\begin{equation}
\alpha_k=\begin{cases} 
{1\over 2} &\hbox{ if }k=2 e_\ell, \ \ell\in\{1\ldots, d\}\\
0 &\hbox{ otherwise;}
\end{cases}
\end{equation}
that is, the only non-zero interactions are those at distance $2$. The corresponding energies are not coercive. Indeed, they have additional ground states, e.g., those given by the checkerboard functions $v$ and $-v$, where
$v_i=(-1)^{\|i\|_1}$. The interpolations $v_\e$ of the corresponding scaled functions $v^\e$ do not converge strongly locally in $L^1(\mathbb R^d)$.
\end{remark}

\begin{example}[non-exact reachability of rational zonotopes by coercive systems]\rm
Let $d=2$ and let $\varphi(z)=8(|\langle e_1+e_2,z\rangle|+|\langle e_1-e_2,z\rangle|)$, corresponding to $\alpha_k=1$ for $k\in\{e_1+ e_2, e_1-e_2, -e_1+e_2, -e_1-e_2\}$, and $\alpha_k=0$ elsewhere, which is not coercive, again with ground states $v$ and $-v$ as in Remark \ref{co-si}.
\end{example}

We give a definition of connectedness related to an Ising system $\{\alpha_k\}$.

\begin{definition} We say that $i$ and $j\in\mathbb Z^d$ are {\em connected with respect to $\{\alpha_k\}$}, or that $i$ and $j$ are {\em $\{\alpha_k\}$-connected}, if there exist $N\in\mathbb N$ and $\{k_1,\ldots, k_N\}$ such that $\sum_{\ell=1}^N k_\ell= j-i$, and $\alpha_{k_\ell}>0$ for all $\ell\in\{1,\ldots, N\}$. We say that the system $\{\alpha_k\}$ is {\em connected} if all $i$ and $j\in\mathbb Z^d$ are connected with respect to $\{\alpha_k\}$.
\end{definition}

\begin{remark}\rm Note that the following statements are equivalent

\ (i) $\{\alpha_k\}$ is connected;

\ \!(ii) $0$ and $j$ are connected for all $j\in\mathbb Z^d$;

(iii) $0$ and $e_n$ are connected for all $n\in \{1,\ldots, d\}$. 

The only non-trivial implication is that (iii) implies (ii). This is proved e.g.~by induction on $n=\|j\|_1$. If $n=1$ the two statements are the same. If $\|j\|_1=n>1$ then we can write $j=j'+e_m$ for some $j'$ with $\|j'\|_1=n-1$ and some $m$ . By the inductive hypothesis there exist $N\in\mathbb N$ and $\{k_1,\ldots, k_N\}$ such that $\sum_{\ell=1}^N k_\ell= j'$, and $\alpha_{k_\ell}>0$ for all $\ell\in\{1,\ldots, N\}$, and there exist $N_m\in\mathbb N$ and $\{k^m_1,\ldots, k^m_{N_m}\}$ such that $\sum_{\ell=1}^{N_m} k^m_\ell= e_m$, and $\alpha_{k^m_\ell}>0$. Then the claim is proven by writing $j=\sum_{\ell=1}^N k_\ell+\sum_{\ell=1}^{N_m} k^m_\ell$.
\end{remark}

We now give a necessary and sufficient condition for a rational zonoid to be obtained from a coercive Ising system.

\begin{theorem}\label{coerns}
Let $\mu$ be a symmetric positive measure on $S^{d-1}$ generating a rational zonoid. Then there exists a coercive Ising system $\{\alpha_k\}$ generating $\mu$ if and only if the set $\{k\in \mathbb Z^d: \mu\big({k\over \|k\|}\big)>0\}$ 
spans the whole $\mathbb Z^d$ on $\mathbb Z$.
\end{theorem}

\begin{proof}
Let $\{\alpha_k\}$ be an Ising system generating $\mu$. Note that we may assume that $\alpha_k>0$ for all points in $k\in \mathbb Z^d$ such that $\mu\big({k\over \|k\|}\big)>0$ since this assumption does not influence $\mu$ and increases the connectedness of $\{\alpha_k\}$. We then have
$$
\Big\{k\in \mathbb Z^d: \mu\Big({k\over \|k\|}\Big)>0\Big\}=\{k\in \mathbb Z^d: \alpha_k>0\},
$$ 
and note that the span of this set is just the set of all finite sums of points $k_\ell$ with $\alpha_{k_\ell}>0$; that is, the set of all points $\{\alpha_k\}$-connected with $0$.

If this set is not the whole $\mathbb Z^d$, then we consider the function defined by
$$
u_i=\begin{cases}
1 & \hbox{ if $i$ is $\{\alpha_k\}$-connected with $0$}\cr
-1 & \hbox{ otherwise.}
\end{cases}
$$
Note that we have $E(u)=0$, but $u$ is not a constant, so that $\|Du\|>0$ and hence the Ising system generating $\mu$ is not coercive.
 
 Conversely, if the set is the whole $\mathbb Z^d$, in particular it contains $\{e_1,\ldots, e_d\}$. Then if  $i,j\in\mathbb Z^d$ are such that $\|i-j\|=1$ and $u_i\neq u_j$, using the $\alpha_{k}$-connectedness there exist $\{k_\ell\}$ with $\sum_{\ell=1}^N k_\ell =j-i$ and $\alpha_{k_\ell}>0$. Writing $i_n=\sum_{\ell=1}^n k_{\ell}$ and $i_0=i$, we have $\sum_{n=1}^N(u_{i_n}-u_{i_{n-1}})=u_j-u_i\neq 0$, and there exist $n\in\{1,\ldots,N\}$ such that $u_{i_n}-u_{i_{n-1}}\neq 0$. Since $k_{i_n}= k_{i_n- i_{n-1}}$ is such that $\alpha_{k_{i_n}}>0$ and the family of all such $\{k_\ell\}$ is finite we deduce that there exists a constant $C>0$ such that $\alpha_{k_{i_n}}(u_{i_n}-u_{i_{n-1}})^2\ge C$. These indices $i_{n}$ and $i_{n-1}$ may be shared by a number of pairs $(i,j)$ bounded by $(\sum_\ell \|k_\ell\|)^d$, so that we can bound $\#\{ (i,j): u_i\neq u_j\}$ by the energy, and the coerciveness of $\{\alpha_k\}$ follows.
\end{proof}

We now examine non-degenerate non-coercive Ising systems and show that they can be seen as a superposition of a finite number of coercive Ising systems.

\begin{remark}[discrete-to-continuum convergence to multiple parameters]\label{co-no}\rm
Let $\{\alpha_k\}$ be an Ising system with symmetric $\alpha_k\ge0$ satisfying \eqref{decay}.
If the system $\{\alpha_k\}$ is non-degenerate, 
the set \begin{equation}\label{subL}
\mathcal L=\{i\in\mathbb Z^d: i\hbox{ is } \{\alpha_k\}\hbox{-connected to }0\}
\end{equation}
 is a $d$-dimensional Bravais sublattice of $\mathbb Z^d$. If the system is not coercive, we can consider the equivalence classes $\mathbb Z^d/\mathcal L$, which are a finite number $M$, that we can represent as $\mathcal L_\ell=m_\ell+\mathcal L$ for $\ell\in\{1,\ldots, M\}$.  If $v^\e\colon \e\mathcal L_\ell\to \{-1,1\}$ then we can define the convergence $v^\e\to A_\ell$ on the (translated) lattice $\e\mathcal L_\ell$ as the convergence of the piecewise-constant interpolations $v^\ell_\e$ on $\e\mathcal L_\ell$ to $2\chi_{A_\ell}-1$ in $L^1_{\rm loc}(\mathbb R^d)$. Thanks to the $\{\alpha_k\}$-connectedness of $\mathcal L_\ell$ the functionals $E_\e$ are coercive with respect to this convergence, so that if $u^\e$ is a sequence with $\sup_\e E_\e(u^\e)<+\infty$, up to subsequences, we can suppose that, denoting by $u^{\e,\ell}$ the restrictions of $u^\e$ to $\e\mathcal L_\ell$, the corresponding piecewise-constant interpolations $u^\ell_\e$ on $\e\mathcal L_\ell$ converge to $2\chi_{A_\ell}-1$. This defines a convergence $u^\e\to (A_1,\ldots, A_M)$ with respect to which the functionals $E_\e$ are coercive. 
 
 Note that at the same time the piecewise-constant interpolations $u_\e$ of $u^\e$ on $\e\mathbb Z^d$ converge weakly in $L^1_{\rm loc}(\mathbb R^d)$ to 
$u={1\over M}\sum_{\ell=1}^M (2\chi_{A_\ell}-1)$, while the stronger convergence in  Theorem \ref{teo-hom-sys} implies that $A_1=\cdots=A_M$.\end{remark}

The convergence in the previous remark allows to generalize the $\Gamma$-convergence result as follows.

\begin{theorem}[$\Gamma$-convergence to multiple parameters]\label{co-mu-pa}
 Let $\{\alpha_k\}$ be an Ising system with symmetric $\alpha_k\ge0$ satisfying \eqref{decay},
and such that the lattice $\mathcal L$ defined in \eqref{subL} be a $d$-dimensional Bravais sublattice of $\mathbb Z^d$. Then the family $E_\e$ is equicoercive with respect to the convergence $u^\e\to (A_1,\ldots, A_M)$ in Remark {\rm\ref{co-no}} and the $\Gamma$-limit with respect to that convergence is
$$
F_{\mathcal L}(A_1,\ldots, A_M)={1\over M} \sum_{\ell=1}^M F(A_\ell),
$$
where $F=F_\varphi$ is given by \eqref{perimeterF} with $\varphi$ as in \eqref{fi}.
\end{theorem}

\begin{proof}

The functional $E_\e$ can be written as a sum of functionals $E^\ell_\e$ defined on $v\colon \e\mathcal L_\ell\to\{-1,1\}$  by
\begin{equation}\label{eell}
E^\ell_\e(v)=\sum_{i,j\in\mathcal L_\ell}\e^{d-1} \alpha_{j-i} (v_j-v_i)^2,
\end{equation}
with the usual notation $v_i=v(\e i)$.
%
Since $E_\e$ converge in the sense of Theorem \ref{teo-hom-sys} to the functional $F$ therein, we note that each of these functionals $E^\ell_\e$ $\Gamma$-converges to ${1\over M} F$ with respect to the corresponding convergence. Hence, the claim of the theorem follows.\end{proof}

The next theorem is an immediate consequence of the characterization of  support functions of zonoids through their generating measures. Note however that optimal approximation of zonoids is a delicate problem (see e.g.~\cite{Bourgain}).

\begin{theorem}[approximate reachability of zonoids]\label{teo-approx-sys} 
Let $\varphi$ be a support function of a zonoid. Then for every $\eta>0$ there exists a coercive Ising system with a limit energy density $\varphi_\eta$ and Wulff shape a rational zonotope such that 
\begin{equation}\label{approx-phi}
\max\{|\varphi_\eta(\nu)-\varphi(\nu)|: \nu\in S^{d-1}\}<\eta.
\end{equation}
\end{theorem}

\begin{proof} Let $\{\alpha^\eta_{k}\}$ be an Ising system parameterized by $\eta$, and $\varphi_\eta$ the corresponding energy function.
Note that we can write
$$
\varphi(z)=\int_{S^{d-1}}|\langle z,\nu\rangle|d\mu_\eta,
$$
where 
$$
\mu_\eta=2\sum_{k\in\mathbb Z^d}\alpha^\eta_k\|k\|\delta_{ {k\over\|k\|}}.
$$
The convergence
$$
\lim_{\eta\to 0}\max\{|\varphi_\eta(\nu)-\varphi(\nu)|: \nu\in S^{d-1}\}=0
$$
is then implied by the weak$^*$ convergence of $\mu_\eta\to \mu$, where $\mu$ is a generating measure for $\varphi$. The existence of such $\mu_\eta$ is then ensured by the weak$^*$ density of finite sums of Dirac deltas. We may also suppose that $\alpha_k\ge\eta$ for $k\in\{e_1,\ldots,e_d\}$, up to adding a term $\sum_n\eta \delta_{e_n}$ whose weak$^*$ limit is the null measure. The claim then follows after reparameterizing the measures $\mu_\eta$.
\end{proof}

\subsection{Convergence of Wulff shapes to rational zonoids}
If $\varphi(\nu)>0$ for all $\nu\in S^{d-1}$ the Wulff shape $W_\varphi$ of $\varphi$ as defined in \eqref{WuSh} 
admits a variational characterization,  as the minimizer symmetric with respect to $0$ of
\begin{equation}
\min\{F(A): |A|=|W_\varphi|\},
\end{equation}
where the functional $F$ is as in \eqref{perimeterF}. The condition $\varphi>0$ is necessary and sufficient in order that $|W_\varphi|>0$. If this condition is not satisfied on the whole $S^{d-1}$ by the convexity of $\varphi$, either $\varphi$ is identically $0$ or $\varphi>0$ on a $d'$-dimensional space with $d'<d$, and we can consider it as defined on $\mathbb R^{d'}$. We will then restrict to the case that $\varphi(\nu)>0$ for all $\nu\in S^{d-1}$.

With this variational characterization in mind, given a homogeneous Ising system $\{\alpha_{k}\}$ generating a function $\varphi$ we can define a {\em discrete Wulff shape for} $\{\alpha_{k}\}$ as any $u^\e$ solution of the minimum problem
\begin{equation}
\min\Bigl\{ E_\e(u): \#\{i: u_i=1\}= N_\e\Big\},
\end{equation}
where $N_\e\in\mathbb N$ is such that $\e^dN_\e$ tends to $|W_\varphi|$; e.g.,~$N_\e=\big\lfloor{1\over\e}|W_\varphi|^{1\over d}\big\rfloor^d$.

The following result shows the relation between rational zonoids and discrete Wulff shapes.

\begin{theorem}[Rational zonoids as limits of discrete Wulff shapes]
Let $\varphi$ be the energy density of the limit $F$ of the Ising system $\{\alpha_k\}$ as in Theorem {\rm\ref{teo-hom-sys}}, and let $W$ be the corresponding related rational zonoid. Suppose that $\varphi(\nu)>0$ for all $\nu\in S^{d-1}$; then there exists a family $u^\e$ of discrete Wulff shapes such that $u^\e\to W$.
\end{theorem}

\begin{proof} If the system $\{\alpha_k\}$ is coercive then this result is a consequence of the $\Gamma$-convergence of $E_\e$ to $F$, the fact that, by a translation argument, we can suppose that discrete Wulff shapes are bounded and with barycenter at distance at most of order $\e$ from $0$, and that recovery sequences $u^\e$ for a set $A$ with $|A|=|W_\varphi|$ can be taken with $\#\{i: u^\e_i=1\}= N_\e$.

If the system is not coercive, by the condition $\varphi>0$ it must nevertheless be non-dege\-ne\-rate.  As in the proof of Theorem \ref{co-mu-pa} the functional $E_\e$ can be written as the sum of the functionals $E^\ell_\e$ in \eqref{eell}, sing the same notation as in Remark \ref{co-no} for the sets $\mathcal L_\ell$. Since each $\mathcal L_\ell$ is disconnected from $\mathcal L_{\ell'}$ if $\ell \neq\ell'$, we can decompose the minimum problem as
\begin{eqnarray*}
&&
\min\Bigl\{ E_\e(u): \#\{i: u_i=1\}= N_\e\Big\}\\
&&
=\min\bigg\{\sum_{\ell=1}^M \min \Bigl\{ E^\ell_\e(u): \#\{i\in\mathcal L_\ell: u_i=1\}= N^\ell_\e\Big\}: \sum_{\ell=1}^M N^\ell_\e= N_\e\bigg\}.
\end{eqnarray*}
We can suppose that $\e^dN^\ell_\e\to \lambda_\ell$ for each $\ell\in\{1,\ldots, M\}$,
with $\sum_{\ell=1}^M\lambda_\ell= |W_\varphi|$.
Since the unit cell of $\mathcal L$ has measure $M$, we have
\begin{eqnarray*}
\lim_{\e\to 0}\min \Bigl\{ E^\ell_\e(u): \#\{i\in\mathcal L_\ell: u_i=1\}= N^\ell_\e\Big\}&=& {1\over M}F\Big(M {\lambda_\ell\over |W_\varphi|}W_\varphi \Big)\\
&=&M^{d-1}{\lambda^{d-1}_\ell\over |W_\varphi|^{d-1}}F(W_\varphi ),
\end{eqnarray*}
and
\begin{eqnarray*}
&&
\lim_{\e\to 0}\min\Bigl\{ E_\e(u): \#\{i: u_i=1\}= N_\e\Big\}\\
&&
=\min\bigg\{\sum_{\ell=1}^M M^{d-2} {\lambda^{d-1}_\ell\over |W_\varphi|^{d-1}}F(W_\varphi ): \sum_{\ell=1}^M{\lambda_\ell\over  |W_\varphi|}=1\bigg\}= F(W_\varphi).
\end{eqnarray*}
In the last equality we have used the convexity of the $(d-1)$-th power, which implies that $\lambda_\ell={1\over M}|W_\varphi|$ for all $\ell$. This also implies that, using the arguments for coercive systems, we can take all interpolations of minimizers $u^\e_\ell$ in each $\e\mathcal L_\ell$ converging to $W_\varphi$. The corresponding $u^\e$ give discrete Wulff shapes converging to $W_\varphi$. 
\end{proof}

\subsection{Asymptotic surface tension of an Ising system}
The simplest variational way to associate an energy density to an Ising system is by computing the average limit surface energy; that is, the energy necessary to have a transition from a state $1$ to a state $-1$ through an hyperplane oriented with a normal $\nu$. This can be done for more general non-homogeneous Ising systems. To that end, given non-negative coefficiants $\{a_{ij}\}$ we define a localized energy on a cube $TQ^\nu$, where $T>0$ and $Q^\nu$ is a unit cube centered in $0$ with two faces orthogonal to $\nu$, as follows:
$$
E(u,TQ^\nu)= {1\over T^{d-1}} \sum_{i\hbox{ or }j\in \mathbb Z^d\cap TQ^\nu}a_{ij}(u_i-u_j)^2.
$$
Note that, if we set $\e={1\over T}$, this can be interpreted as the part of the energy \eqref{fegen} `contained in the cube $Q^\nu$'. 

In order to impose boundary conditions, due to the non-local nature of the energies we have to fix the values of functions outside $TQ^\nu$. To that end, we consider minimum problems of the form
\begin{equation}
m_T(\nu)=\min\Bigl\{ E(u,TQ^\nu): u_i=\pm 1 \hbox{ if }\pm\langle i, \nu\rangle >0\hbox{ for all }i\not\in TQ^\nu\Bigl\}.
\end{equation}
This value can be considered as the minimum value of the transition from $-1$ to $1$ around the hyperplane $\Pi^\nu=\{z\in \mathbb R^d: \langle z,\nu\rangle=0\}$.

\begin{definition}[surface tension of an Ising system]
The {\em surface tension} of the Ising system $\{a_{ij}\}$ is defined as 
\begin{equation}\label{surtengen}
\varphi(\nu)=\liminf\limits_{T\to+\infty} {1\over T^{d-1}}m_T(\nu).
\end{equation}
\end{definition} 

We note that the definition of surface tension does not require any condition on the coefficients $a_{ij}$ except their non-negativity. In the case of coefficients $a_{ij}=\alpha_{i-j}$ a straightforward computation gives the  formula for $\varphi$.

\begin{proposition}[surface tension of a homogeneous Ising system]
If $\alpha_k\ge 0$ for all $k\in\mathbb Z^d$ and $a_{ij}=\alpha_{i-j}$ then the {\em surface tension} $\varphi$ of the Ising system is given by  
\begin{equation}
\varphi(\nu)= 4 \sum_{k\in\mathbb Z^d\setminus \{0\}} \alpha_k|\langle k,\nu \rangle|.
\end{equation}
Furthermore, the $\liminf$ in \eqref{surtengen} is a limit.
\end{proposition}

\begin{proof}
With fixed $k\in\mathbb Z^d$ with $\langle \nu,k\rangle\neq0$, we note that for any test function $u$ and any line $L_{i,k}=\{i+tk: t\in\mathbb R\}$ with $i\in\mathbb Z^d$ and such that $L_{i,k}\cap TQ^\nu\cap \Pi^\nu\neq \emptyset$, there exist and least one index $n\in\mathbb Z$ such that $u_{i+nk}\neq  u_{i+(n+1)k}$. This implies that $$m_T(\nu)\ge 4T^{d-1} \sum_{\|k\|\le K}\alpha_k|\langle k,\nu\rangle| +O(T^{d-2})$$ for every fixed $K$, and the lower bound letting $K\to+\infty$. An upper bound is simply given taking $u_i=1$ if $\langle i, \nu\rangle \ge0$ and $u_i=-1$ if $\langle i, \nu\rangle <0$ for $i\in\mathbb Z^d$. A direct computation shows that this is a minimizing sequence and the existence of the limit in \eqref{surtengen}.
\end{proof}

\subsection{Directed Ising systems and non-centered rational zonoids}\label{directed}
We conclude this section with a generalization of Ising systems, where the energies take the form
 \begin{equation}\label{fegen-dir}
E_\e(u)=\sum_{i,j\in\mathbb Z^d}\e^{d-1} \,\alpha_{i-j}((u_i-u_j)^+)^2,
\end{equation}
where $t^+$ indicates the positive part of $t\in\mathbb R$. In this case the interaction between two points $i$ and $j$ such that $u_i=1$ and $u_j=-1$ is taken into account with the coefficient $\alpha_{i-j}$, while if $u_j=1$ and $u_i=-1$ with the coefficient $\alpha_{j-i}$. This is a particular case of the inhomogeneous directed Ising systems studied in \cite{CK}.

For energies \eqref{fegen-dir} we do not assume that $\alpha_{-k}=\alpha_k$ in order not to lose in generality. Nevertheless, the proof of the convergence in Theorem \ref{teo-hom-sys} works essentially unchanged, with the limit energy density given by
\begin{equation}\label{Ising-8}
\varphi(z) =  4\sum_{k\in\mathbb Z^d} \alpha_{k} \langle k,z\rangle^+.
\end{equation}
Note that in the perimeter functional \eqref{perimeterF} the integration is done with $\nu$ the inner normal to the set $A$, which may reflect the asymmetry of the Ising system.

If the range of $\alpha_k$ is finite, then the Wulff shape of the function $\varphi$ is 
\begin{equation}\label{WuSh}
W_\varphi=\Bigl\{w\in\mathbb R^d\colon  \hbox{ there exist } s_k\in[0,1]\hbox{ such that } w=4\sum_{k\in\mathbb Z^d} s_k\alpha_{k} k\Bigr\};
\end{equation}
that is, $W_\varphi$ is the finite sum of the segments $[0,w_\ell]$ in $\mathbb R^d$, where the set $\{w_\ell\}$ coincides with the set of  $4\alpha_{k} k$  such that $\alpha_k>0$. This is the translation of a (centered) rational zonotope by the vector ${1\over2}\sum_\ell w_\ell$.  Proceeding as in Theorem \ref{teo-hom-sys} we deduce that directed Ising systems correspond to all translations of (centered) rational zonoids, and then, using Theorem \ref{teo-approx-sys}, that all (non-centered) zonoids are reached by sequences of zonotopes generated by directed Ising systems.

\section{Connections with discrete-to-continuum homogenization}
We now highlight the connection between the definitions given until now and general results on periodic Ising systems. This will allow us to define a generalization of rational zonoids.

\smallskip

The limit in Theorem {\rm\ref{teo-hom-sys}} can be interpreted as a particular case of homogenization of periodic Ising systems. We say that an Ising system $\{a_{ij}\}$ is {\em periodic with period $N$} if we have
\begin{equation}\label{per-1}
a_{i+Ne_n\,j+Ne_n} =a_{ij}
\end{equation}
for all $i,j\in\mathbb Z^d$ and $n\in\{1,\ldots, d\}$, which in turn is equivalent to
\begin{equation}\label{per-2}
a_{i+k\,j+k} =a_{ij}
\end{equation}
for all $i,j\in\mathbb Z^d$ and $k\in N\mathbb Z^d$.
 If $N=1$ then this condition is equivalent to requiring that $a_{ij}= \alpha_{i-j}$ for some $\alpha_k$, so that homogeneous Ising systems coincide with Ising systems periodic with period $1$.

For periodic systems we have the following result.

\begin{theorem}[homogenization and crystallinity of periodic Ising systems \cite{BP,CK}]\label{cryst}
Let $\{a_{ij}\}$ be an $N$-periodic Ising system satisfying
\begin{equation}\label{coerN}
\max_{i\in \mathbb Z^d}\sum_{j\in\mathbb Z^d}a_{ij}\|j-i\|<+\infty.
\end{equation}
Then there exists the $\Gamma$-limit in the sense of Theorem {\rm\ref{teo-hom-sys}}. If in addition the system is with finite range; that is, there exists a constant $K$ such that  $a_{ij}=0$ if $\|i-j\|>K$ then the Wulff shape of the limit functional $F$ is a polytope.
\end{theorem}

\subsection{Zonotopes generated by homogenized Ising systems}
In this section we observe that the definition of rational zonotope can be generalized in view of Theorem \ref{cryst}, as the following definitions.

\begin{definition}\label{N-zo}
We say that $W$ is a {\em rational zonotope of order $N$} if the corresponding $\varphi$ is the limit of an $N$-periodic Ising system with finite range.
We say that $W$ is  a {\em rational zonoid of order $N$} if the corresponding $\varphi$ is the limit of an $N$-periodic Ising system satisfying \eqref{coerN}. 
We say that $W$ is  a {\em zonoid of order $N$} if it is the limit in the Hausdorff distance of rational zonotopes of order $N$.
\end{definition}

\begin{remark}\rm
 Rational zonotopes and zonoids of order $1$ are  rational zonotopes and zonoids as defined above.
 Note that condition \eqref{coerN} corresponds to \eqref{decay} if $N=1$.
 
The analysis in \cite[Proposition 2.9]{CK} imply that the energy density $\varphi$ of a  rational zonotope of order $N$ is differentiable outside rational directions, which suggests that $(d-1)$-dimensional faces of Wulff shapes should have normals in rational directions. 
\end{remark}

\begin{proposition} If $W$ is a rational zonoid of order $N$, then it is the limit in the Hausdorff distance of rational zonotopes of order $N$.
\end{proposition}

\begin{proof}
If $W$ is a rational zonoid of order $N$ generated by $\{a_{ij}\}$ it suffices to consider the rational zonotopes $W_n$ of order $N$ generated by $\{a^n_{ij}\}$, where $a^n_{ij}=a_{ij}$ if $\|i-j\|\le n$ and $a^n_{ij}=0$ if $\|i-j\|> n$.
\end{proof}

We finally show that the union of all zonoids of order $N$ is dense in the class of all symmetric convex sets.

\begin{theorem}[density of zonotopes of order $N$ as $N\to+\infty$]\label{dens-N}
For every convex bounded open set $W$ symmetric with respect to the origin there exist zonotopes $W_k$ of order $N_k$ converging to $W$.
\end{theorem}

\begin{proof} Let $\varphi$ be the support function of $W$.
In \cite{BK} the following result is proved: if $0<\alpha\le\beta$ are such that
$$\alpha\|z\|_1\le \varphi(z)\le\beta\|z\|_1,$$
then there exist periodic systems $\{a^k_{ij}\}$ of period $N_k$ with $a^k_{ij}\in\{\alpha,\beta\}$ such that the related homogenized energy densities $\varphi_k$ converge to $\varphi$. By Theorem \ref{cryst} the Wulff shapes $W_k$ are zonotopes of order $N_k$ that converge to $W$.
\end{proof}

\subsection{Further possible generalizations}
While homogeneous directed Ising systems as in Section \ref{directed} only involve a translation in the resulting generated rational zonoid, the class of homogenized periodic directed Ising systems could be strictly larger than that of translations of the non-directed analog. We can therefore consider periodic Ising systems with coefficients $a_{ij}$ satisfying \eqref{per-1}, \eqref{per-2}, and \eqref{coerN}, and the corresponding energies 
\begin{equation}\label{fegen-dirN}
E_\e(u)=\sum_{i,j\in\mathbb Z^d}\e^{d-1} \,a_{ij}((u_i-u_j)^+)^2.
\end{equation}
Note again that we do not suppose that $a_{ij}=a_{ji}$, a condition that would not be restrictive for un-directed Ising systems. The results in \cite{CK} ensure the validity of the claims of Theorem {\rm\ref{cryst}}.
We can therefore give definitions of directed rational zonotope and zonoid of order $N$ as in Definition \ref{N-zo}. It would be interesting to know if such zonoids still possess a center of symmetry, in which case it is likely that a result as Theorem \ref{dens-N} holds for all $W$ with a center of symmetry. 

\smallskip
Finally, we note that another possible class of perimeter functionals are those generated by perturbed periodic Ising systems of the form 
\begin{equation}
E_\e(u)=\sum_{i,j\in\mathbb Z^d}\e^{d-1} \alpha_{i-j} (u_i-u_j)^2+\sum_{i\in\mathbb Z^d}\e^d u_ig_i,
\end{equation}
where $g$ is a periodic function with zero average,
and small enough so that $E_\e(u)$ remains non-negative on bounded
configurations. This corresponds to adding a volume term with zero average.
These energies still converge to a perimeter functional, whose form may depend on the perturbation. 
A link with the homogenization of directed Ising systems can be obtained following the results in \cite{CK} as done for the continuous analog in \cite[Sec.~4]{CT}.

\bigskip

\noindent{\bf Acknowledgements.}
 This paper is based on work supported by the National Research Project (PRIN  2017BTM7SN) 
``Variational Methods for Stationary and Evolution Problems with Singularities and 
 Interfaces", funded by the Italian Ministry of University and Research. 
Andrea Braides is a member of GNAMPA, INdAM.

\end{document}